\documentclass[
 reprint,
 amsmath,amssymb,
10ptb,
]{revtex4-2}

\usepackage{graphicx}
\usepackage{dcolumn}
\usepackage{bm}
\usepackage{tikz}

\DeclareMathOperator{\PEG}    {PE_G}
\DeclareMathOperator{\PE}    {PE}
\DeclareMathOperator{\DE}    {DE}
\DeclareMathOperator{\DEG}    {DE_G}
\DeclareMathOperator{\row}    {row}
\DeclareMathOperator{\NCDF}    {NCDF}
\DeclareMathOperator{\MIX}    {MIX}
\DeclareMathOperator{\RGG}    {RGG}
\DeclareMathOperator{\RGGs}    {RGGs}

\newcommand{\set}[2]{\{ \, #1 \, | \, #2 \, \} } 
\newcommand{\card}[1]{\lvert#1\rvert}

\newcommand{\map}[3]{ #1 \colon #2 \rightarrow #3}

\newcommand{\G}{{G}}
\newcommand{\R}{\mathbb{R}} 
\newcommand{\V}{\mathcal{V}} 
\newcommand{\E}{\mathcal{E}} 
\newcommand{\Na}{\mathbb{N}}
\newcommand{\A}{\mathbf{A}}
 
\newcommand{\X}{\mathbf{X}} 
\newcommand{\N}{\mathbb{N}}

\newcommand{\DG}{ \overrightarrow{G}} 
\newcommand{\Y}{\mathbf{Y}} 
\newcommand{\y}{\mathbf{y}}

\usepackage{bbm}
\usepackage{amsthm}

\newtheorem{proposition}{Proposition}
\usepackage{enumitem}
\usepackage{float}
\usepackage{graphicx}
\usepackage{subcaption}
\usepackage[font={footnotesize}]{caption}
\DeclareMathOperator{\DEGD}    {DE_{\overrightarrow{G}}}

\begin{document}

\preprint{APS/123-QED}

\title{Dispersion entropy: A Measure of Irregularity for Graph Signals}

\author{John Stewart Fabila-Carrasco$^1$, Chao Tan$^2$, and Javier Escudero$^1$}
\affiliation{%
$^1$ School of Engineering, Institute for Digital Communications, University of Edinburgh, West Mains Rd, Edinburgh, EH9 3FB, UK.\\
$^2$ School of Electrical and Information Engineering, Tianjin University, Tianjin 300072, China
	}

\date{\today}

\begin{abstract}
We introduce a novel method, called Dispersion Entropy for Graph Signals, $\DEG$, as a powerful tool for analysing the irregularity of signals defined on graphs. We demonstrate the effectiveness of $\DEG$ in detecting changes in the dynamics of signals defined on synthetic and real-world graphs, by defining mixed processing on random geometric graphs or those exhibiting with small-world properties. 
Remarkably, $\DEG$ generalises the classical dispersion entropy for univariate time series, enabling its application in diverse domains such as image processing, time series analysis, and network analysis, as well as in establishing theoretical relationships (i.e., graph centrality measures, spectrum). 
Our results indicate that $\DEG$ effectively captures the irregularity of graph signals across various network configurations,  successfully differentiating between distinct levels of randomness and connectivity. Consequently, $\DEG$ provides a comprehensive framework for entropy analysis of various data types, enabling new applications of dispersion entropy not previously feasible, and revealing relationships between graph signals and its graph topology.
\end{abstract}

\maketitle
\textbf{Introduction.} 
Entropy is a fundamental tool for assessing irregularity and non-linear behaviour in data. Permutation entropy ($\PE$) is an effective algorithm for capturing dynamics in time series (1D data)~\cite{Bandt2002} and has been widely used in finance, physics, and biology~\cite{cao2004detecting}. However, $\PE$ considers only the order of values, discarding important amplitude information. Dispersion Entropy ($\DE$) was introduced to overcome this limitation~\cite{Rostaghi2016}, and has since been applied to EEG analysis~\cite{azami17} and rotary machines~\cite{rostaghi2019application}. 

The growing availability of data defined on complex networks, such as social networks~\cite{huang2018graph}, transportation systems~\cite{ortega18}, and industrial processes~\cite{fabila2022noise}, has driven interest in extending entropy metrics from time series to more general domains. Recently, $\PE$ has been extended to analyse images (2D data)~\cite{Morel2021} and irregular domains (graphs)~\cite{Fabila-Carrasco2022}. While $\DE$ has been defined for 2D data~\cite{azami19b}, there is no existing $\DE$ algorithm for analysing data defined on graphs. Such an extension would enable analysis of real-world systems with graph-based structure where classical $\DE$ was not previously applicable, providing a powerful framework for data analysis across a wide range of applications in Graph Signal Processing (GSP)~\cite{ortega18}. 

Smoothness is a fundamental property extensively studied in GSP~\cite{ortega18,shuman2013emerging,dong2016learning}, typically through the use of the combinatorial Laplacian's quadratic form. Intuitively, a graph signal is considered smooth if connected vertices exhibit similar values~\cite{dong2016learning}. Nonetheless, this definition may not fully capture the complex dynamics of graph signals due to its relationship with the spectrum~\cite{stankovic2019introduction}. To address this limitation, we propose in this letter a novel method, based on classical $\DE$ for time series, which effectively captures the irregularity of graph signals, providing critical insights into the underlying graph structure and data.

To evaluate our method's performance, we employed synthetic and real-world graphs, including random geometric graphs (used to model wireless sensor networks~\cite{kenniche2010random}) and small-world networks (observed widely in biological systems~\cite{watts98}, social networks~\cite{newman2001random}, and complex systems~\cite{newman2000models}). In our analysis, we generalised the mix process $\MIX(p)$, a stochastic process combining a sinusoidal signal with random dynamics controlled by the parameter $p\in[0,1]$. This process has been employed to assess the performance of various entropy metrics in time series~\cite{Pincus1994, Azami2018} and images~\cite{Silva2016}. Moreover, we analyse centrality measures, which assign ranking values to the graph's vertices based on their position or importance within the graph. Centrality measures play a crucial role in social network analysis for evaluating the importance of vertices in communication~\cite{borgatti2006graph,das2018study}. 

\emph{Contribution.}  In this letter, we propose a method for defining Dispersion Entropy for Graph Signals, denoted as $\DEG$.  Our approach generalises the classical univariate definition of $\DE$ by incorporating topological information through the adjacency matrix. We demonstrate the effectiveness of $\DEG$ on synthetic and real-world datasets, and characterise the relationship between graph topology and signal dynamics. Our results indicate that $\DEG$ is a promising technique for analysing graph data, holding potential for numerous applications in fields such as biomedicine and social sciences.

\emph{Notation.} 
A \emph{simple undirected graph} $G$ is defined as a triple $G = (\V,\E,\A)$, where $\V$ is a finite set of vertices (without isolated vertices), $\E$ is the set of edges, and $\A$ is the corresponding adjacency matrix. A \emph{graph signal} is a real function defined on the vertices $\map{\X}{\V}{\R}$, represented as an $N$-dimensional column vector, $\X=\left[x_1, x_2, \dots, x_N \right]^T \in \R^{N\times 1}$, with the same indexing as the vertices. The combinatorial Laplacian and normalised Laplacian are denoted by $\Delta$ and \textbf{$L$}, respectively.

A \emph{$d$-dimensional Random Geometric Graph} ($\RGG$) is a graph in which each vertex $v_i \in \mathcal{V}$ is assigned a random $d$-dimensional coordinate $v_i \rightarrow \mathbf{x}_i=(x_i^1,\dots,x_i^d) \in [0,1]^d$. Two vertices $v_i,v_j \in \mathcal{V}$ are connected by an edge if the distance between their assigned coordinates is below a predefined threshold $r>\textbf{d}(v_i,v_j)$ (see~\cite{dall2002random}).

\textbf{Dispersion Entropy for Graph Signals ($\DEG$).} Let $\X$ be a graph signal defined on $\G$, $2\leq m\in\Na$ be the \emph{embedding dimension}, $L\in\Na$ be the \emph{delay time} and $c\in\Na$ be the \emph{class number}. The $\DEG$ is defined as follows: 
\begin{enumerate}[wide, labelwidth=!, labelindent=0pt]
	\item The \emph{embedding matrix} $\textbf{Y}\in\R^{N\times m}$ is given by
		$\textbf{Y}=[\textbf{y}_0,  \textbf{y}_1,  \cdots,  \textbf{y}_{m-1}]$, defined by
	\[\textbf{y}_k=D\A^{kL} \textbf{X}\in \R^{N\times 1}\;, \quad k=0,1,\dots,m-1\;,\]
	 where $D$ is the diagonal matrix $D_{ii}=1/\sum_{j=1}^N (\A^{kL})_{ij}$. 
	
	\item \emph{Map function.} Each entry of the embedding matrix $\Y$ is mapped to an integer number from $1$ to $c$, called a class. The function $\map{F}{\R}{\N_c}$ where $\N_c=\{1,2,\dots,c\}$ is applied element-wise on the matrix $\Y$, i.e. $F(\Y)\in \N_c^{N\times m}$ where $F(\Y)_{ij}=F(y_{ij})$.

	\item \emph{Dispersion patterns.} Each row of the matrix $F(\textbf{Y})$, called an \emph{embedding vector}, is mapped to a unique dispersion pattern.
	Formally, the \emph{embedding vectors} consist of $m$ integer numbers (ranged from $1$ to $c$) corresponding to each row of the matrix $F(\Y)$, i.e., $\row_i(F(\Y))=\left( F(y_{ij})\right)_{j=1}^m $ for $i=1,2,\dots,N$. The set of dispersion patterns is $\Pi=\set{\pi_{v_{1} v_{2} \dots v_{m}}}{v_i\in  \N_c}$.
	Each embedding vector is uniquely mapped to a dispersion pattern, i.e.,  $\row_i(F(\Y))\rightarrow \pi_{v_{1} v_{2} \dots v_{m}}$ where $v_1=F(y_{i1}), v_2=F(y_{i2}),\dots, v_m=F(y_{im})$.
	\item \emph{Relative frequencies.} For each dispersion pattern $\pi\in\Pi$, its relative frequency is obtained as:
	\begin{equation*}
		p\left(\pi\right)=\frac{\card{\left\{i \mid i \in \V,\row_i(F(\Y)) \text { has type } \pi\right\}}}{N}\;.
	\end{equation*}
	\item The \emph{Dispersion Entropy for Graph Signals} $\DEG$ is computed as the normalised Shannon's entropy for the distinct dispersion patterns as follows:
	\begin{equation*}
		\DEG(\X,m,L,c)=-\dfrac{1}{\log(c^m)}\sum_{\pi \in \Pi } p(\pi) \ln p(\pi)\;.
	\end{equation*}		
\end{enumerate}

The $\DEG$ algorithm offers several unique features and properties. The \emph{embedding matrix} is a key component that captures the topological relationships between the graph and signal. With a chosen embedding dimension $3\leq m \leq7$, and delay time commonly set to $L=1$ (values suggested~\cite{Bandt2002}), the embedding matrix $\textbf{Y}\in\R^{N\times m}$ is constructed. 
Each column vector $\textbf{y}_k$ is calculated by averaging the signal values of neighbouring vertices, i.e. $\textbf{y}_k=D\A^{kL}\X$, where the power of the adjacency matrix $\A^{kL}$ denotes the number of $kL$-walks between two vertices. Additionally, the diagonal matrix $D$ serves as a normalisation factor. The first column of the matrix $\textbf{Y}$ corresponds to the original graph signal, i.e., $\y_0=\X$, and the second column is related to the normalised Laplacian $L$, specifically, $\y_1=\X-L\X$.
 
\emph{Map functions.} To address limitations in assigning the signal $\X$ to only a limited number of classes, various maps functions $\map{F}{\R}{\N_c}$ have been proposed~\cite{Rostaghi2016}. The non-linear cumulative distribution function ($\NCDF$) is commonly utilised. The map $\map{G}{(0,1)}{\N_c}$ is defined as $G(x)=round(cx+0.5)$, where rounding increases or decreases a number to the nearest digit. The map $\map{\NCDF}{\R}{(0,1)}$ is defined as:
\[\NCDF(x)=\frac{1}{\sigma \sqrt{2 \pi}} \int_{-\infty}^{x} e^{\frac{-(t-\mu)^2}{2 \sigma^2}} \mathrm{d} t\]
where $\mu$ and $\sigma$ represent the mean and standard deviation of $\X$, respectively. Thus, $\map{F=G\circ \NCDF}{\R}{\N_c}$ is the map function used in our implementation of the $\DEG$ algorithm.

\emph{Dispersion patterns.} The number of possible dispersion patterns that can be assigned to each embedding vector is $c^m$. Moreover, the number of embedding vectors constructed in the $\DEG$ algorithm is $N$, one for each vertex. In contrast, classical $\DE$ has a number of embedding vectors dependent on the parameters $m$ and $L$, specifically, $n-(m-1)L$.

\emph{Shannon’s entropy} provides a measure of irregularity that can be used to compare signals defined on different graphs. The value of Shannon's entropy ranges from $0$ (regular behaviour) to $1$ (irregular behaviour).

\textbf{Dispersion entropy for directed graphs.}
The algorithm $\DEG$ provides a tool for analysing undirected graph signals, and can be extended to \emph{directed graphs} with minor modifications. Additionally, the algorithm can be applied to any graph signal, but for time series, it produces the same values as the classical $\DE$~\cite{Rostaghi2016}. This is established in Proposition~\ref{prp:equal}.

\begin{proposition}[\emph{Equivalence of $\DE$ and $\DEG$ for time series}]\label{prp:equal} 
	Let $\textbf{X}=\left\{x_i\right\}_{i=1}^{N}$ be a time series and  $\overrightarrow{G}=\overrightarrow{P}$ is the directed path on $N$ vertices. Then, for all $m, c$ and $L$ the following equality holds:
	\[\DE(m,L,c)=\DE_{\overrightarrow{P}}(m,L,c)\;.\]
\end{proposition}
\begin{proof}
	Please refer to the supplemental material~\cite{supplemental}.
\end{proof}

\textbf{MIX Processing on $\RGGs$.} 
We introduce a graph-based stochastic process $\MIX_G(p)$ defined on $\RGGs$ to assess the performance of $\DEG$ in capturing complex signal dynamics. Here, $G$ is a $d$-dimensional $\RGG$ with $N$ vertices, and the graph signal $\MIX_G(p)$ is defined by:
\begin{equation}\label{eq:mix}
	\MIX_G(p)_i = (1-R_i)S_i + R_iW_i\quad \text{for} \quad 1\leq i \leq N\;,
\end{equation}
where $R_i$ is a random variable with a probability $p$ of taking the value $1$ and a probability $1-p$ of taking the value $0$, $W_i$ is uniformly distributed white noise sampled from the interval $[-\sqrt{3}, \sqrt{3}]$, and $S_i=\sum_{j=1}^{d} \sin(f x_i^j)$ represents a sinusoidal signal with frequency $f$. 

The construction of a $d$-dimensional $\RGG$ requires selecting two parameters, $r$ and $N$, while the graph signal generated by the $\MIX_G(p)$ process incorporates random noise (determined by $p$) into some values of the sinusoidal signal (determined by $f$). Our algorithm, $\DEG$, detects changes in the frequency of the signal (increasing $f$), the presence of white noise (increasing $p$), and the graph connectivity (increasing $r$) by increasing the entropy values of $\DEG$. Fig.~\ref{fig:RGG} illustrates the effectiveness of $\DEG$ in detecting the dynamics of the $\MIX_G$ process.

\begin{figure}

		\centering
	\includegraphics[scale=.22345]{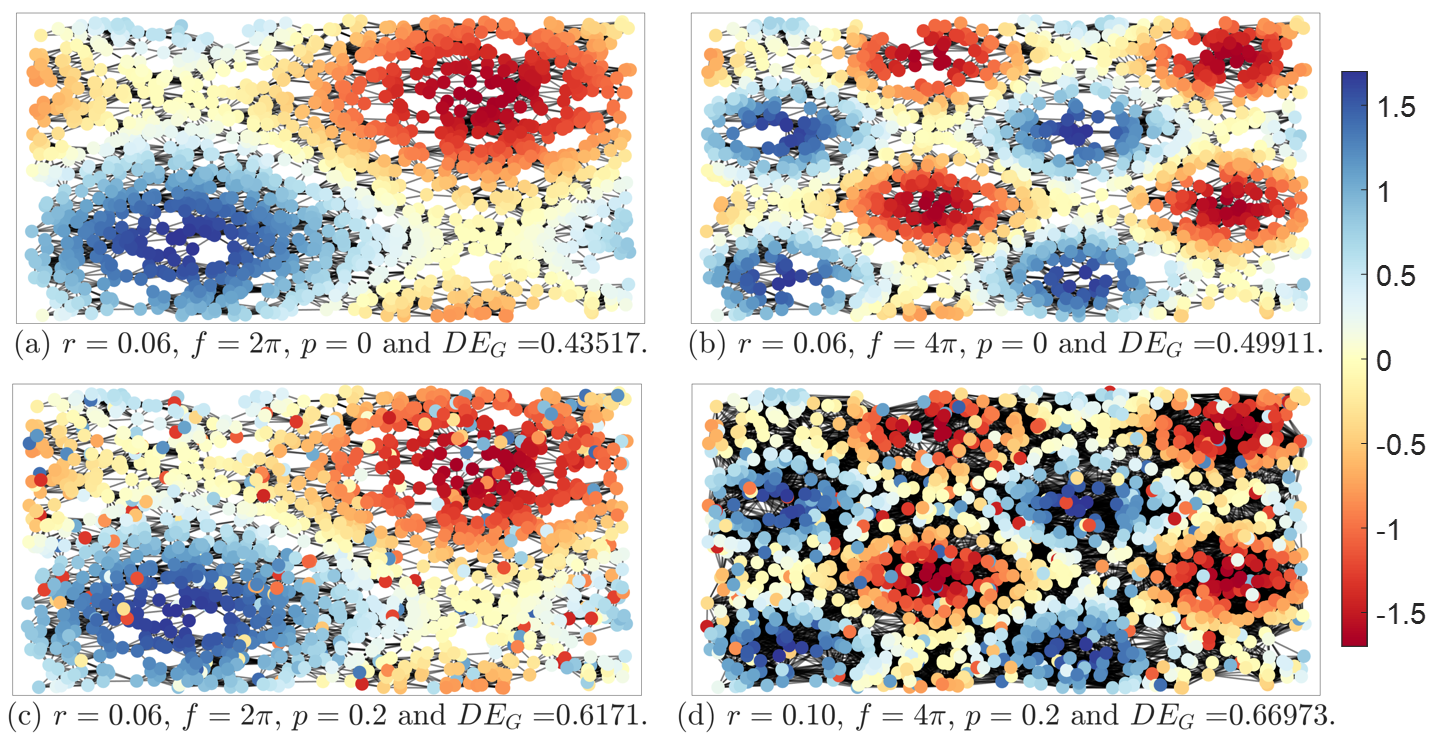} 	
		\caption{Examples of $\RGGs$ with $N=1,500$ and values $r=0.06$ and $r=0.10$. The graph signals are generated by the $\MIX_G$ process with different parameter values.}
		\label{fig:RGG} 
	\end{figure}

\emph{Fixing the graph, changing the signal.} We analyse the impact of different parameter values on the irregularity of the graph signal $\MIX_G(p)$ by fixing the underlying $\RGG$ with constant $N=1500$ and $r=0.06$. We employ a fixed embedding dimension of $m=3$, the number of classes set at $c=3$, time delay $L=1$, and $\NCDF$ as the non-linear map (similar results are obtained for others non-linear mappings and values of $m$, $c$, and $L$). 

Increasing the frequency parameter $f$ of the $\MIX_G(p)$ process results in a more irregular graph signal.  The frequency $f=2\pi$ and $f=4\pi$ of the sine function in Eq.~\ref{eq:mix} are depicted in Fig.\ref{fig:RGG}a)-b). This increase in frequency produces more variation in the graph signal values between neighbouring vertices. Our algorithm $\DEG$ detects these dynamics by increasing the entropy values. Similarly, an increase in the randomness parameter $p$ results in a more random signal. The parameters $p=0$ and $p=0.2$ in Eq.~\ref{eq:mix} are depicted in Fig.\ref{fig:RGG}a), c). The $\DEG$ algorithm detects the change in randomness, by increasing the entropy values.

More generally, we compute the entropy values for a range of frequencies from $3/2\pi$ to $16\pi$, as well as for different levels of noise, with probabilities ranging from $0$ to $1$. The results of 30 realizations are depicted in Fig.~\ref{fig:freq_p1}, showing the mean and standard deviation. The $\DEG$ algorithm effectively detects the increasing irregularity of the signal by increasing the entropy values. Moreover, the algorithm can distinguish between different levels of irregularity in the $\MIX_G(p)$ signal based on the chosen value of $p$. 
\begin{figure}
	\centering
	\begin{subfigure}[b]{0.49\linewidth}
	\includegraphics[scale=.3]{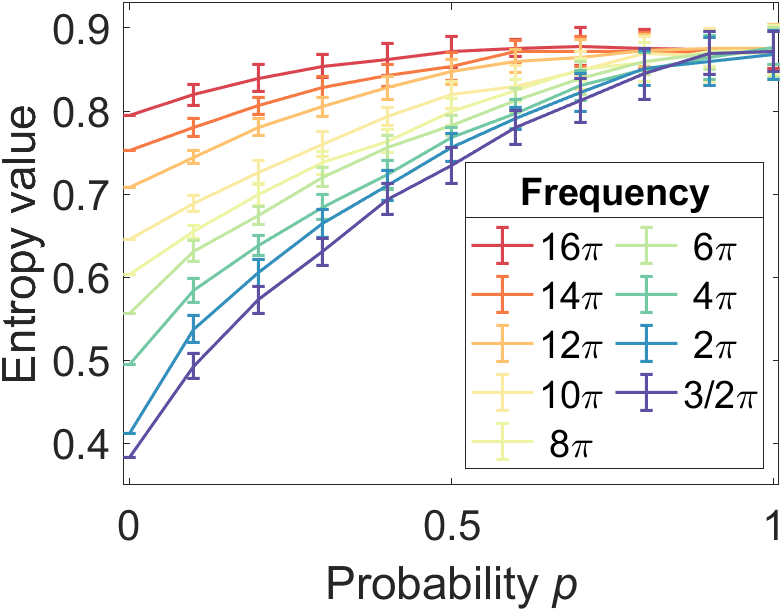} 
	\caption{\label{fig:freq_p1}}
	\end{subfigure}
	\begin{subfigure}[b]{0.49\linewidth}
	\includegraphics[scale=.3]{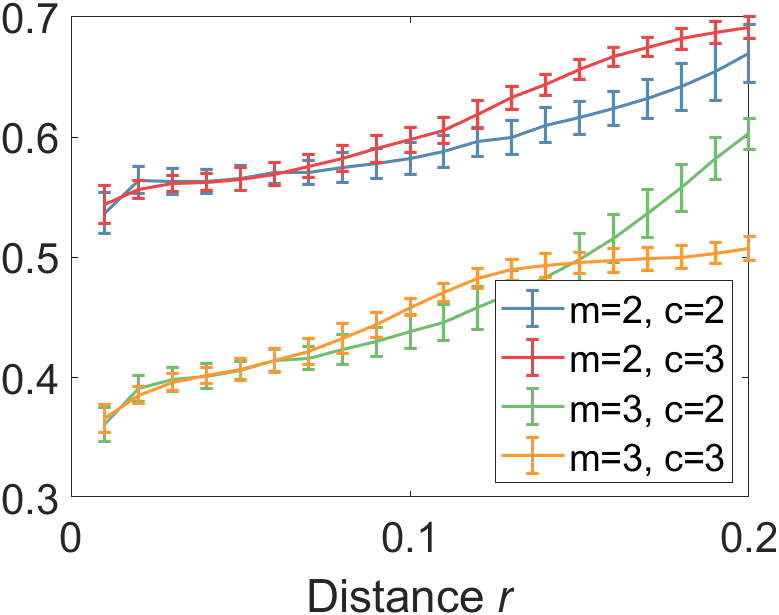}
		\caption{\label{fig:freq_p2}} 
	\end{subfigure}
	\caption{\label{fig:freq_p} Entropy values (a) for a fixed graph, increasing the noise and for several frequencies and (b) the underlying graph is more connected.}
\end{figure}

\emph{Fixing the signal, changing the graph}. By fixing the graph signal, we investigate the behaviour of the $\DEG$ algorithm as the underlying graph changes. Specifically, we examine the impact of increasing the distance parameter $r$ from $0.01$ to $0.3$ used for construct the $\RGG$ with $N=1,500$ vertices. Entropy values are computed for $20$ realisations, and the mean and standard deviation are depicted in Fig.~\ref{fig:freq_p2} for several values of $m$ and $c$. As $r$ increases, the number of edges increases, connecting more distant vertices with different values. The resulting patterns are more irregular, with more changes and a wider distribution, leading to an increase in the entropy value.

\textbf{The spectrum of the Laplacian and $\DE_G$.} Let $\X$ be a graph signal, \emph{the smoothness of $\X$} is given by $\X^T \Delta \X$~\cite{ortega18}. We examine the relationship between $\DEG$ and the spectrum of $\Delta$ acting on an $\RGGs$ (similar results are obtained for other random graphs).  

Let $G$ be a $\RGG$ with $N=1,500$ vertices. The eigenvalues of $\Delta$ and its corresponding eigenvectors are denoted by $\sigma=\{\lambda _{1}\leq \lambda _{2}\leq \cdots \leq \lambda _{N}\}$ and $\{f_i\}_{i=1}^N$, respectively. The smoothness of each eigenvector is evaluated and normalised based on the classical definition, i.e., $\lambda _{N}^{-1} f_i^T \Delta f_i$, and the results are shown in Fig.~\ref{fig:DEGeigenvalues}. Each eigenvector $f_i$ is considered as a graph signal and $\DEG$ is computed for $c=2,3,4$ and $m=2$. The results are depicted in Fig.~\ref{fig:DEGeigenvalues}. The smoothness definition is an increasing function, i.e., smaller eigenvalues correspond to smoother eigenvectors (also known as graph Fourier modes~\cite{girault2018irregularity}). Such information is limited especially when eigenfunctions associated with equal eigenvalues (and equal smoothness) exhibit different levels of irregularity. By applying the $\DEG$ algorithm, we can better understand and analyse the dynamics of these eigenfunctions. 
\begin{figure}
	\centering
	\includegraphics[scale=.2]{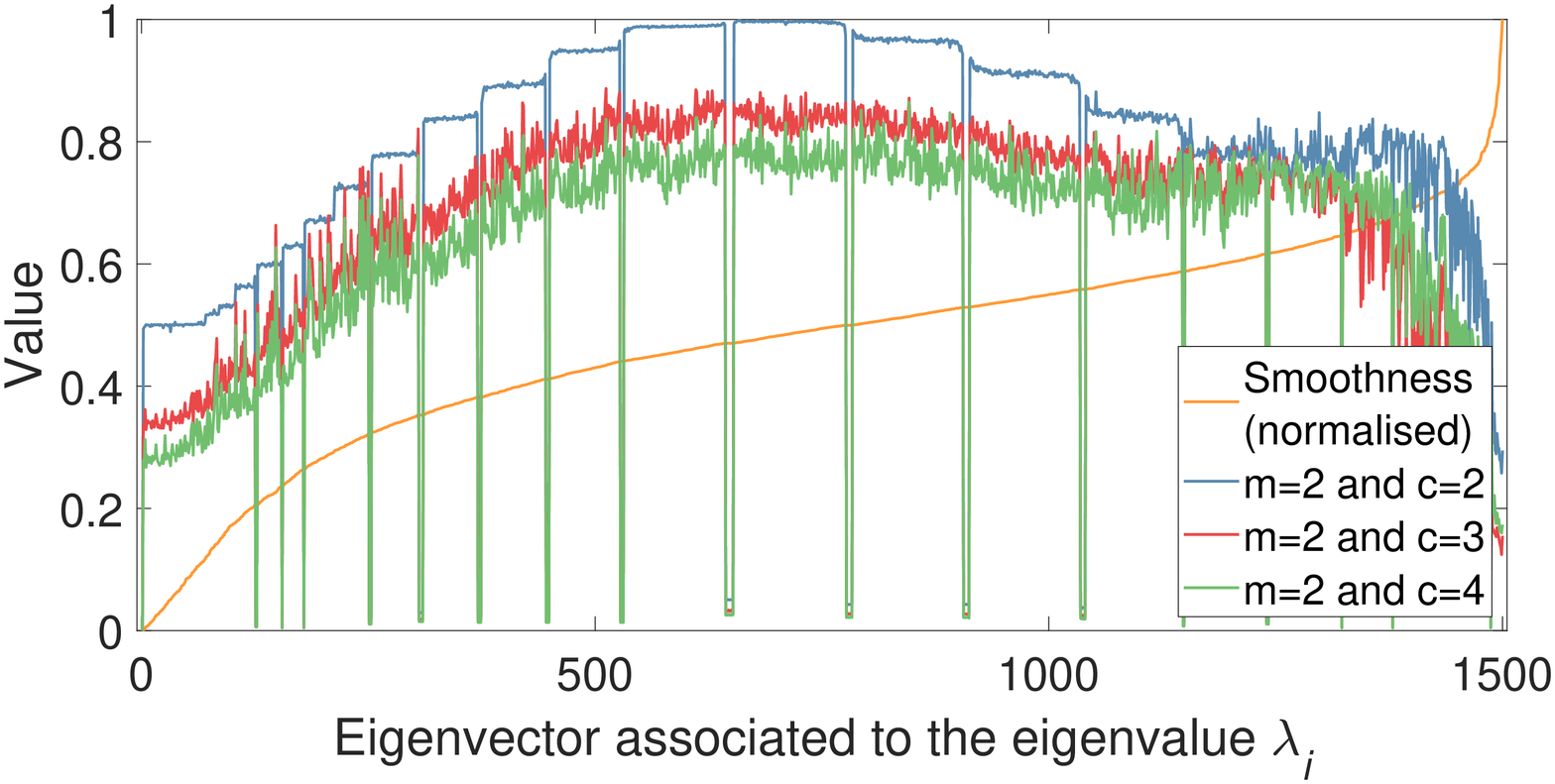} 
	\caption{Entropy values of $\DEG$ and smoothness based on the Laplacian $\Delta$ for the eigenvalues as graph signals.\label{fig:DEGeigenvalues} }
\end{figure}

The dispersion entropy computed for different values of $m$ and $c$ enables us to capture abrupt changes in entropy values when the dynamics of eigenfunctions change. Fig.~\ref{fig:several} depicts six eigenvectors $\{f_j\}_{j=527}^{532}$ corresponding to the eigenvalues $\{\lambda_j\}_{j=527}^{532}$. The definition of smoothness of $f_j$ coincides with the value $\lambda$, and the eigenvalue $\lambda_{528}=15$ has a multiplicity equal to four, and its eigenfunctions $\{\lambda_j\}_{j=528}^{531}$ exhibit a regular behaviour, while $f_{527}$ and $f_{532}$ are more irregular.  Hence, classical definitions are not able to fully capture the difference in dynamics within the graph signals. In contrast, the $\DEG$ algorithm is capable of detecting them. In particular, the entropy value of the eigenfunctions is nearly close to $0$ if the signal exhibits a more regular dynamic and close to $1$ for the most irregular eigenfunctions. Thus, $\DEG$ detects eigenvalues with high multiplicity, useful for the construction of isospectral graphs~\cite{fabila2022geometric}.

\begin{figure}
	\centering
	\includegraphics[scale=.2]{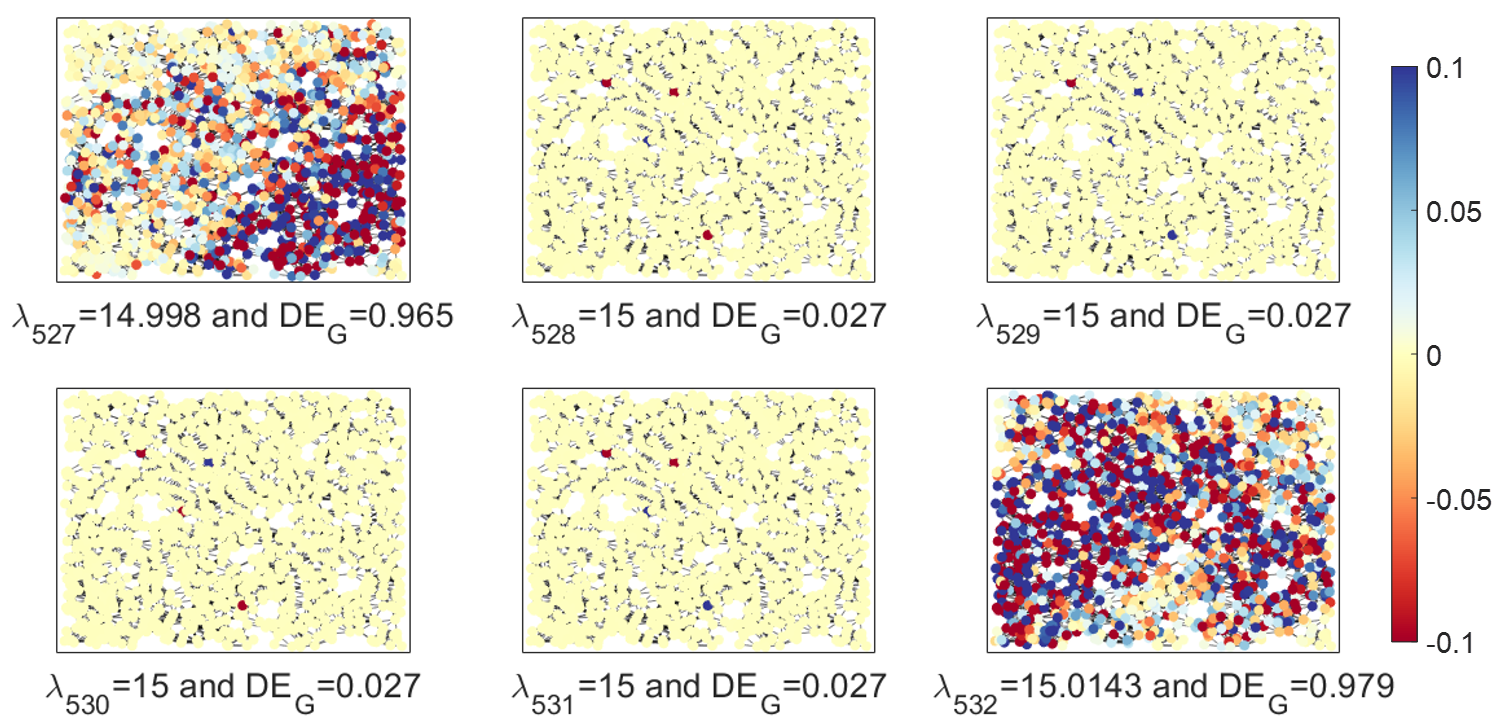} 
	\caption{Several eigenfunctions and their entropy values.\label{fig:several} }
\end{figure}

\textbf{Small-world networks and $\DEG$.} We evaluate the performance of $\DEG$ in detecting dynamics on signals defined on small-world networks, generated by the Watts-Strogatz model~\cite{newman2001random}, and changing the mean degree $k$ and rewiring probability $p$. Let $\G$ be a small-world network with $N=1,500$ and various graph signals, including a random signal, a recurrence relation (logistic map~\cite{Bandt2002}), a stochastic process (Wiener process~\cite{wiener1938homogeneous}), and a periodic signal (sine). 

\emph{Fixing $k$, changing $p$.} By fixing $k=1$, we analyse the effect of the parameter $p$ (ranging from $0$ to $1$) in the construction of the network $G_p$ and the entropy values. We compute $\DEG$ for each graph signal for $20$ realizations, and the mean and standard deviation are depicted in Fig.~\ref{fig:timeseries_p1}. 
For $p=0$, the underlying graph $G_p$ is a cycle of $N$ vertices. A path graph is a geometric perturbation of a cycle~\cite{fabila2022spectral,carrasco2020discrete}
and due to Prop.~1, we can consider the values of $p=0$ to be the classical $\DE$. The classical $\DE$ is able to detect the dynamics of various signals, but its computation does not involve the topological structure, thus it only works for the path graph. In contrast, $\DEG$ takes into account not only the signal information but also the graph structure. In this setting, the dynamic of the random signal is almost constant, because it is not affected by $G_p$. The Wiener process and sine signals exhibit lower entropy values for $p=0$ (e.g., the cycle), as their dynamics stem from either periodicity (sine) or stochastic processes (Wiener). However, as $p$ increases, the underlying graph becomes more random, and hence the entropy value also increases. In any case, $\DEG$ is still able to distinguish the random signal from the periodic signal and the Wiener process (for all $p<0.8$). Two logistic map signals are generated, one with oscillatory behaviour ($r=3.3$) and one with chaotic behaviour ($r=3.7$). These characteristics are is well detected by $\DEG$ for all values of $p$.
\begin{figure}
	\centering
		\centering
	\begin{subfigure}[b]{0.49\linewidth}
		\includegraphics[scale=.3]{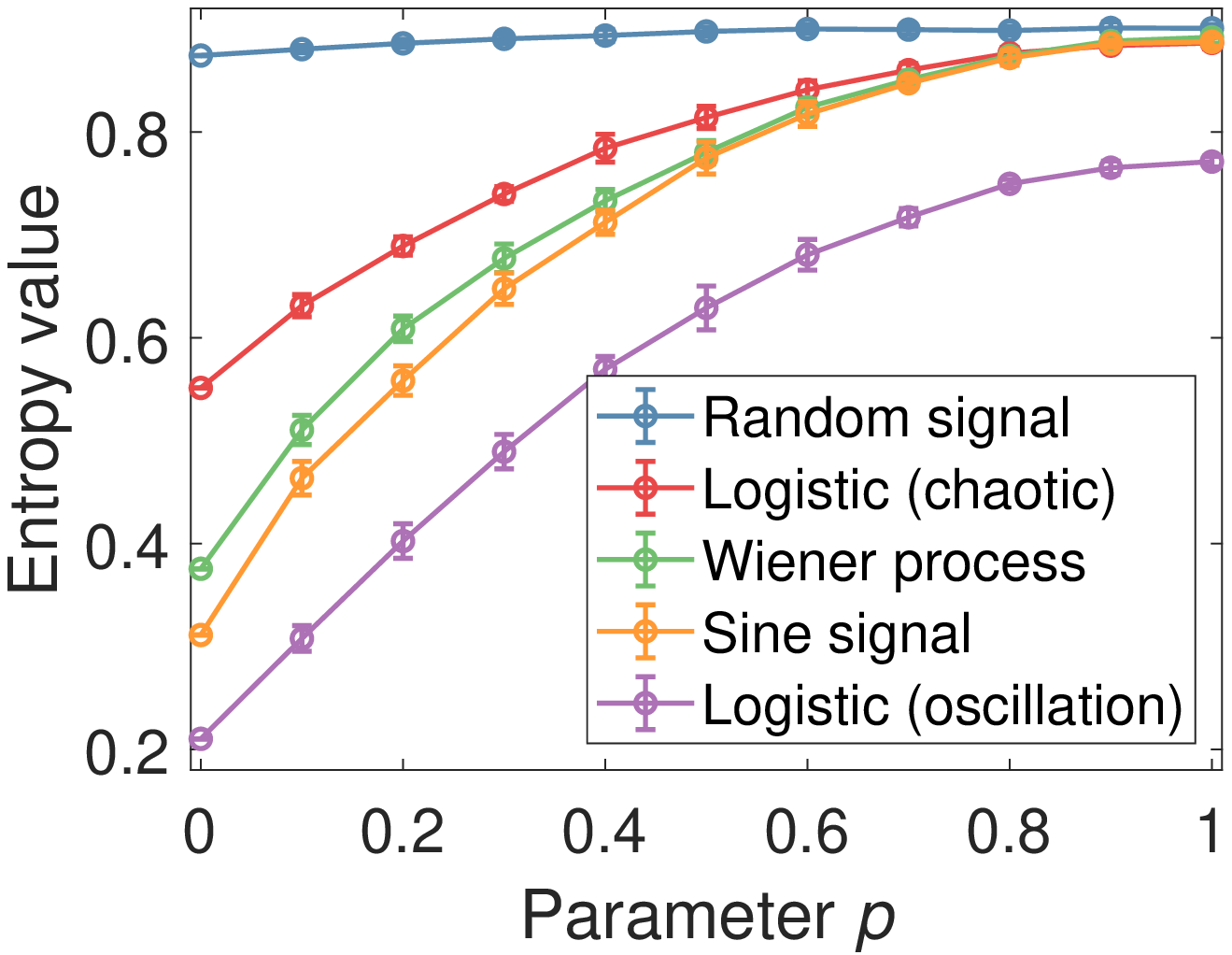} 
		\caption{\label{fig:timeseries_p1} }
	\end{subfigure}
	\begin{subfigure}[b]{0.49\linewidth}
		\includegraphics[scale=.3]{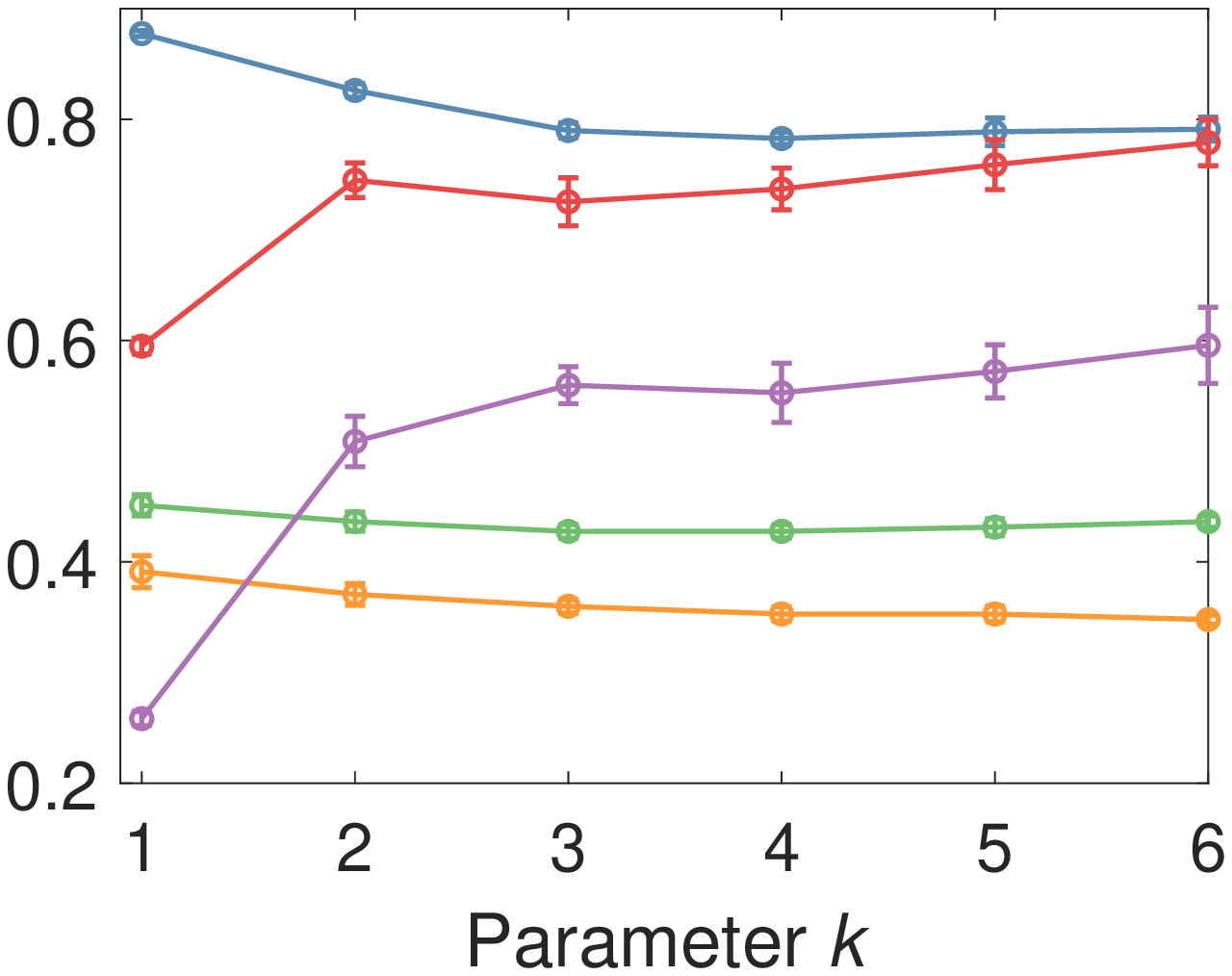}
		\caption{\label{fig:timeseries_p2} } 
 	\end{subfigure}
	\caption{\label{fig:timeseries} Entropy values for different signals defined on a small-world network generated by the Watts-Strogatz model.}
\end{figure}

\emph{Fixing $p$, changing $k$.}  
By fixing $p=0.05$, the underlying graph $G_k$ where $1\leq k\leq 6$ increases the connectivity. In Fig.~\ref{fig:timeseries_p2}, we present the entropy values for each graph signal. The entropy values for the sine and Wiener signals almost remain constant, independent of $G_k$, due to their periodicity and stochastic dynamics. However, the logistic map exhibits a higher degree of variability in its entropy values as $k$ increases. This is because the logistic map is defined by a recurrence formula, where each value depends only on the previous sample, and if $k$ increases, the underlying $G_k$ has more connections between neighbourhoods, which may disrupt the recurrence relation, generating more irregular signals and resulting in higher entropy values. Conversely, the random signal shows a reduction in entropy values as $k$ increases, as the creation of more connections leads to a more robust average value due to the law of large numbers.

\textbf{Graph Centrality Measures and $\DEG$.} Each centrality measure can be considered as a graph signal, allowing the application of the $\DEG$ algorithm to assess the irregularity of centrality measures on real and synthetic graphs (refer to Table~I in the supplemental material~\cite{supplemental}).

We used six centrality measures as graph signals,  namely~\cite{borgatti2006graph,das2018study}: \emph{Eigenvector centrality}, \emph{Betweenness}, \emph{Closeness}, \emph{Harmonic centrality}, \emph{Degree} and \emph{Pagerank}. The $\DEG$ algorithm leverages the graph topology to effectively detect irregularities generated by each centrality measure, as demonstrated in Fig.~\ref{fig:central2}. In particular, the \emph{Eigenvector Centrality} produces smooth signals~\cite{dong2016learning} in most graphs, and this is reflected in low entropy values. Well-connected vertices tend to appear on the shortest paths between other vertices. When the graph has only a few such vertices, the entropy of the \emph{Betweenness} measure is lower.  In cases where the graph has a more irregular distribution of vertices with this characteristic (e.g., in the sphere due to its symmetry), the entropy values are higher. A similar effect occurs when considering the average length of the shortest path between the vertex and all other vertices, as detected by the \emph{Closeness} measure. Finally, the \emph{Degree} and \emph{PageRank} measures produce more irregular graph signals because each signal's value defined on the graph depends only on local properties (the degree or the number and importance of the other vertices connected to it) rather than global properties (such as average paths between vertices in the previous measures). 

\begin{figure}[H]
	\centering
	\includegraphics[scale=.2]{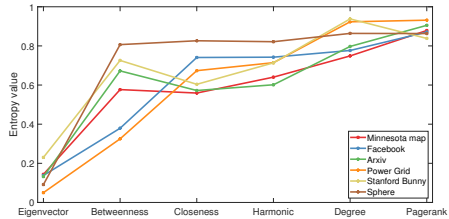}  
	\caption{\label{fig:central2} The dispersion entropy for various centrality measures.}
\end{figure}

\textbf{Comparing $\DEG$ and $\PEG$ Performance} The Permutation Entropy for Graph Signals, denoted by $\PEG$~\cite{Fabila-Carrasco2022}, marked the first entropy metric specifically designed for graph-based data analysis. Both methods rely on the adjacency matrix, but $\PEG$ primarily focuses on the order of amplitude values (local properties), which might result in the loss of valuable information regarding the amplitudes (global properties). $\DEG$ addresses these limitations by providing a more comprehensive way to characterize the dynamics of graph signals. We conducted the same previous analysis with $\PEG$ (supplemental material~\cite{supplemental}), and found that $\DEG$ consistently outperforms $\PEG$ in all cases, highlighting the potential of our novel method for effectively analysing graph signal irregularities.

\textbf{Conclusions.} We have introduced Dispersion Entropy for Graph Signals ($\DEG$), a method that enhances the analysis of irregularities in graph signals. Our approach generalises classical dispersion entropy, enabling its application to a wide array of domains, including real-world graphs, directed and weighted graphs, and unveiling novel relationships between graph signals and graph-theoretic concepts (e.g., eigenvalues and centrality measures). By overcoming the limitations of the classical smoothness definition, $\DEG$ offers a more comprehensive approach to analysing graph signals and holds significant potential for further research and practical applications, as it effectively captures the complex dynamics of signals across diverse topology configurations.

\textbf{ Acknowledgement} This work was supported by the Leverhulme Trust via a Research Project Grant (RPG-2020-158).
\newpage

\newcommand{\beginsupplement}{%
	\setcounter{table}{0}
	\renewcommand{\thetable}{S\arabic{table}}%
	\setcounter{figure}{0}
	\renewcommand{\thefigure}{S\arabic{figure}}%
	\setcounter{section}{0}
	\renewcommand{\thesection}{\Roman{section}}%
	\setcounter{equation}{0}
	\renewcommand{\theequation}{S\arabic{equation}}%
	\setcounter{proposition}{0}
	\renewcommand{\theequation}{S\arabic{proposition}}%
}
\clearpage
\beginsupplement
\setcounter{secnumdepth}{2}
\begin{center}
	\Large {\bf Supplemental Materials for ``Dispersion entropy: A Measure of Irregularity for Graph Signals''}
\end{center}

\section{Dispersion Entropy for Directed Graphs.} 
In the letter, we have introduced the Dispersion Entropy for graph signals, denoted as $\DEG$, in the context of \emph{undirected} graphs. To extend this concept to \emph{directed graphs} or \emph{digraphs}, the approach remains analogous, with the primary distinction being the need to incorporate specific constraints on the rows of the embedding matrix. These constraints are imposed by the well-defined vectors $\y_k$.

Let $\DG = (\V,\E,\A)$ be a digraph with $N$ vertices, where $\A$ denotes the adjacency matrix of the directed graph, and $\textbf{X}=\left\{x_i\right\}_{i=1}^{n}$ is a signal defined on $\DG$. Given an \emph{embedding dimension} $m$ with $2\leq m\in\Na$, a \emph{delay time} $L\in\Na$, and a \emph{class number} $c\in\Na$, the Dispersion Entropy for Directed Graphs ($\DEGD$) is defined as follows: 
\begin{enumerate}[wide, labelwidth=!, labelindent=0pt]
	\item \emph{Embedding matrix.} Let $\V^*\subset \V$ be the set given by:
	\begin{equation*}\label{eq:directedset}
		V^*=\set{i\in \V }{ \sum_{j=1}^n (\A^{kL})_{ij}\neq 0 \text{ for all } k=0,1,\dots,m-1}\:.
	\end{equation*}	
	The \emph{embedding matrix} $\textbf{Y}^*\in\R^{\card{V^*}\times m}$ is given by:

	\begin{align}
		\textbf{Y}^*=[\textbf{y}^*_0,  \textbf{y}^*_1,  \cdots,  \textbf{y}^*_{m-1}]
	\end{align}
	where $\textbf{y}^*_k\in \R^{\card{V^*}\times 1}$, given by the restriction of $\textbf{y}_k$ to the vertices in $V^*$, i.e.,  $\textbf{y}^*_k=\left.\y^k\right|_{V^*}$. 
	\item \emph{Map function.} Each element of the embedding matrix $\Y^*$ is mapped to an integer number from $1$ to $c$, called a class, i.e., we define a function $\map{F}{\R}{\N_c}$ where $\N_c=\{1,2,\dots,c\}$ that applies element-wise on the matrix $\Y^*$, i.e. $F(\Y^*)\in \N_c^{N\times m}$ where $F(\Y^*)_{ij}=F(y^*_{ij})$.

	\item \emph{Dispersion patterns.} Each row of the matrix $F(\textbf{$\Y^*$})$, called an \emph{embedding vector}, is mapped to a unique dispersion pattern.
	Formally, the \emph{embedding vectors} consist of $m$ integer numbers (from $1$ to $c$) corresponding to each row of the matrix $F(\Y^*)$, i.e., $\row_i(F(\Y^*))=\left( F(y^*_{ij})\right)_{j=1}^m $ for $i=1,2,\dots,N$. The set of dispersion patterns is defined as $\Pi=\set{\pi_{v^*_{1} v^*_{2} \dots v^*_{m}}}{v^*_i\in  \N_c}$.
	Each embedding vector is uniquely mapped to a dispersion pattern, i.e.,  $\row_i(F(\Y^*))\rightarrow \pi_{v^*_{1} v^*_{2} \dots v^*_{m}}$ where $v_1=F(y^*_{i1}), v_2=F(y^*_{i2}),\dots, v_m=F(y^*_{im})$.
	
	\item \emph{Relative frequencies.} For each dispersion pattern $\pi\in\Pi$, its relative frequency is obtained as:
	\begin{equation*}
		p\left(\pi\right)=\frac{\card{\left\{i \mid i \in \V,\row_i(F(\Y)) \text { has type } \pi\right\}}}{\card{\V^*}}\;.
	\end{equation*}
	\item \emph{Shannon's entropy.} The \emph{dispersion entropy for graph signals} $\DEGD$ is computed as the normalised Shannon's entropy for the distinct dispersion patterns as follows:
	\begin{equation*}
		\DEGD(\X,m,L,c)=-\dfrac{1}{\log(c^m)}\sum_{\pi \in \Pi } p(\pi) \ln p(\pi)\;.
	\end{equation*}		
\end{enumerate}

\textbf{Properties} The $\DEGD$ algorithm for directed graphs exhibits the following properties:

The directed graph version of $\DEGD$ serves as a generalization of its undirected counterpart. If $G$ is an undirected connected (non-trivial) graph, then $\V^*=\V$, and all the steps remain the same in both the directed and undirected versions of the algorithm. 

The restriction process $\textbf{y}^*_k=\left.\y^k\right|_{V^*}$ is equivalent to the vertex virtualisation process presented in~\cite{fabila2018spectral}.

Similarly, the $\DEGD$ algorithm can be extended to weighted (directed or undirected) graphs by restricting the subset to		\begin{equation*}\label{eq:directedset}
	V^*=\set{i\in \V }{ \sum_{j=1}^n (\textbf{W}^{kL})_{ij}\neq 0 \text{ for all } k=0,1,\dots,m-1}\:.
\end{equation*}	
where $\textbf{W}$ represents the weighted adjacency matrix. This generalisation allows for a more comprehensive analysis of graph signals in various contexts. 
\section{Proof of Proposition 1.} 
The classical dispersion entropy for time series was established in the literature by~\cite{Rostaghi2016}. In the following proposition, we demonstrate that when the $\DEG$ is restricted to time series (considering the directed path as the underlying graph), the $\DEG$ is equivalent to the classical $\DE$.

A \emph{directed path} on $k$ vertices is a directed graph that connects a sequence of distinct vertices with all edges oriented in the same direction, denoted as $\overrightarrow{P}$. Its vertices are given by $\V={1,2,\dots,k}$ and its arcs are $(i,i+1)$ for all $1\leq i \leq k-1$.
\begin{proposition}[\emph{Equivalence of $\DE$ and $\DEG$ for time series}]\label{prp:equal} 
	Let $\textbf{X}=\left\{x_i\right\}_{i=1}^{N}$ be a time series and consider $\overrightarrow{G}=\overrightarrow{P}$ the directed path on $n$ vertices, then for all $m, c$ and $L$, the following equality holds:
	\[\DE(m,L,c)=\DE_{\overrightarrow{P}}(m,L,c)\;.\]
\end{proposition}
\begin{proof}
	The adjacency matrix for the directed path $\A$ is given by 
	\[   
	\A_{ij} = 
	\begin{cases}
		1 &\text{if } \quad i=1,2,\dots,N-1 \quad \text{and} \quad j=i+1  \;,\\
		0 &\text{otherwise }\;. \\
	\end{cases}
	\]
	
	For any $k\in\Na$, the matrix $\A^{k}$ is given by 
	\[   
	\left( \A^k\right) _{ij} = 
	\begin{cases}
		1 &\text{if } \quad i=1,2,\dots,N-k \quad \text{and} \quad j=i+k  \;,\\
		0 &\text{otherwise }\;, \\
	\end{cases}
	\]
	in particular, for all $k=0,1,\dots,m-1$ 	\[   
	\sum\limits_{j=1}^N (\A^{kL})_{ij} = 
	\begin{cases}
		1 &\text{if } \quad i=1,\dots, N-(m-1)L\;,\\
		0 &\text{otherwise } \;.\\
	\end{cases}
	\]
	
	Thus, we have
	\begin{align*}
		\textbf{y}^*_k&=\left.\y_k\right|_{V^*}=\left.D\A^{kL}\textbf{X}\right|_{V^*} \\
		&=[x_{1+kL}, x_{2+kL}, \dots,x_{i+kL}, \dots ,x_{N-(m-1)L}]^T\;.
	\end{align*}
	The \emph{embedding matrix} is given by:
	\[\Y^*=\begin{pmatrix}
		x_1 & x_{1+L}  & \dots & x_{1+(m-1)L}\\
		x_2 & x_{2+L}  & \dots & x_{2+(m-1)L}\\
		\vdots  & \vdots & \ddots & \vdots\\
		x_{N-(m-1)L} & x_{N-(m-2)L} & \dots & x_{N}\\
	\end{pmatrix}\;,\]
	
	and, given a map function $\map{F}{\R}{\N_c}$ defined by $\map{F=G\circ \NCDF}{\R}{\N_c}$, the matrix $F(\Y^*)$ is given by:
	\[F(\Y^*)=\begin{pmatrix}
		z_1 & z_{1+L}  & \dots & z_{1+(m-1)L}\\
		z_2 & z_{2+L}  & \dots & z_{2+(m-1)L}\\
		\vdots & \vdots & \ddots  & \vdots\\
		z_{N-(m-1)L} & z_{N-(m-2)L} &  \dots & z_{N}\\
	\end{pmatrix}\;.\]
	
	Subsequently, the embedding vectors are represented as $\row_i(F(\Y^))=\left(z_i,z_{i+L},\dots z_{i+(m-1)L}\right)$. Due to the fact that $\card{\V^}=N-(m-1)L$, the \emph{relative frequencies} and \emph{Shannon's entropy} associated with the graph-based dispersion entropy ($\DEGD$) and the classical dispersion entropy ($\DE$) are identical.
\end{proof}
\section{Graphs used for analysing centrality measures.}
\begin{table}[h]
	\caption{\label{tab:table1} 
	}
	\begin{ruledtabular}
		\begin{tabular}{lccc}
			\textrm{\textbf{Underlying Graph}}&
			\textrm{$\card{\V}$}&
			\textrm{$\card{\E}$}&
			\textrm{\textbf{Reference}}\\
			\colrule
			Minnesota road network	& 2,642 & 3,303 & \cite{gleich2015matlabbgl} \\
			Social circles: Facebook	& 3,959 & 84,243 &  \cite{leskovec2012learning}\\
			Arxiv GR-QC collaboration & 5,241 & 14,484 &  \cite{leskovec2007graph}\\
			The US power grid	&4,941  & 6,594 & \cite{watts98} \\
			Pointcloud (Stanford Bunny) &2,503  &  13,726& \cite{turk1994zippered} \\
			Sphere &4,000  &  22,630& \cite{perraudin2014gspbox} \\
		\end{tabular}
	\end{ruledtabular}
\end{table}
\section{Comparing $\DEG$ and $\PEG$ Performance}
In this section, we demonstrate the superior performance of the Dispersion Entropy for Graph Signals ($\DEG$) over the Permutation Entropy for Graph Signals, denoted by $\PEG$~\cite{Fabila-Carrasco2022}. By applying both algorithms to all the examples in the manuscript, we consistently observe that $\DEG$ outperforms $\PEG$, highlighting the potential and efficacy of $\DEG$ for analysing graph signal irregularities.

Following the same setting used to produced Fig.~2, 3, 5 and 6 in the manuscript, we substitute $\PEG$ for $\DEG$. The results are depicted in Fig.~\ref{fig:freq_pPEG},~\ref{fig:DEGeigenvaluesPEG},~\ref{fig:timeseriesPEG} and ~\ref{fig:central2PEG}, respectively.

\textbf{Random Graphs and $\PEG$.} The $\PEG$ algorithm is not able to detect increasing of the signal irregularity (due to frequency increments) and is unable to differentiate between distinct levels of irregularity in the $\MIX_G(p)$ signal based on the parameter $p$ (Fig.~\ref{fig:freq_p1PEG}). Similarly, in Fig.~\ref{fig:freq_p2PEG}, as graph connectivity increases (by raising $r$) the algorithm saturates for an embedding dimension of $m=2$. To achieve accurate characterisations, it is necessary to increase $m>2$  and even that, the behaviour is not monotonous, whereas $\DEG$ performs well with smaller embedding dimensions.    
\begin{figure}[h]
	\centering
	\begin{subfigure}[b]{0.49\linewidth}
		\includegraphics[scale=.3]{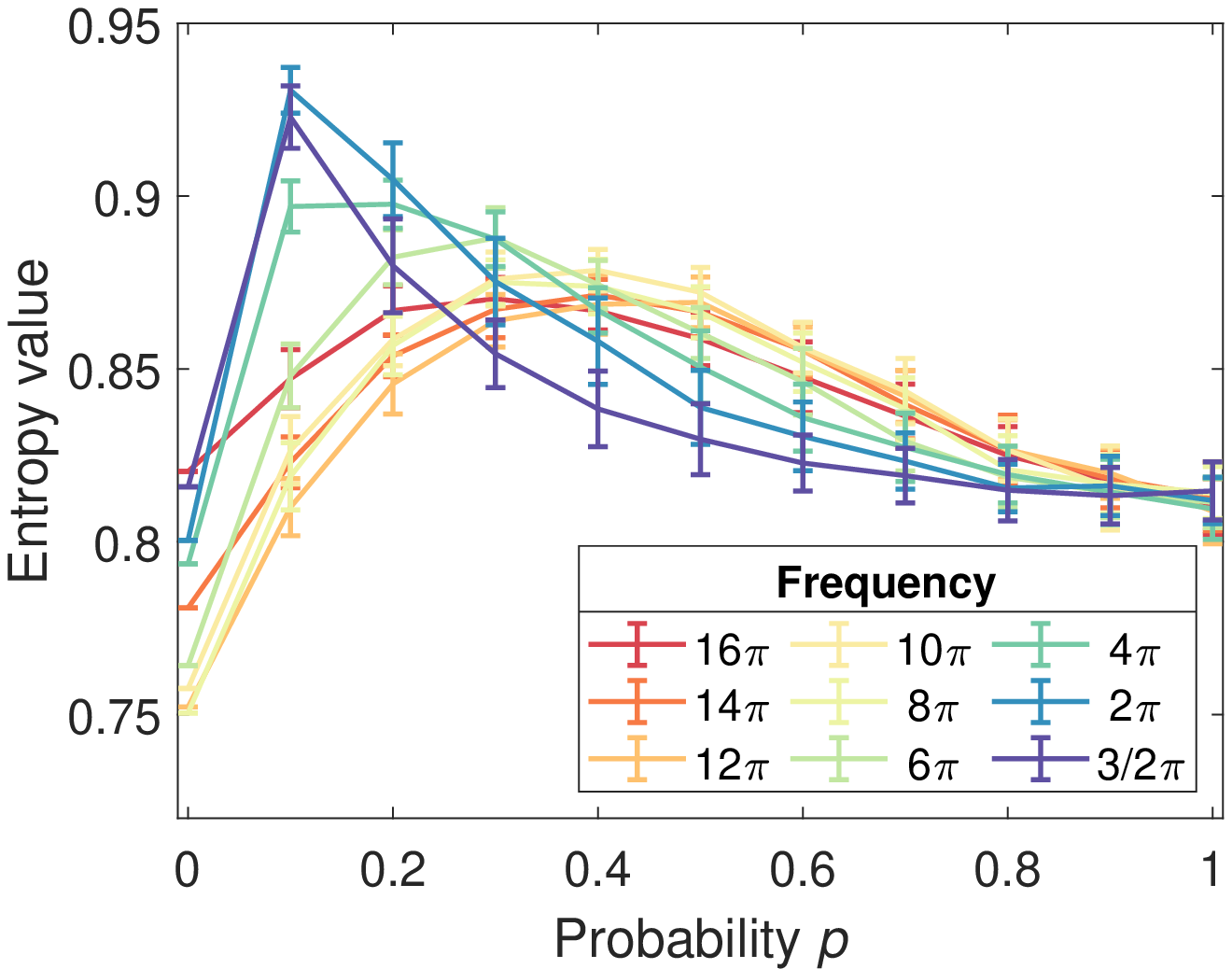} 
		\caption{\label{fig:freq_p1PEG}}
	\end{subfigure}
	\begin{subfigure}[b]{0.49\linewidth}
		\includegraphics[scale=.3]{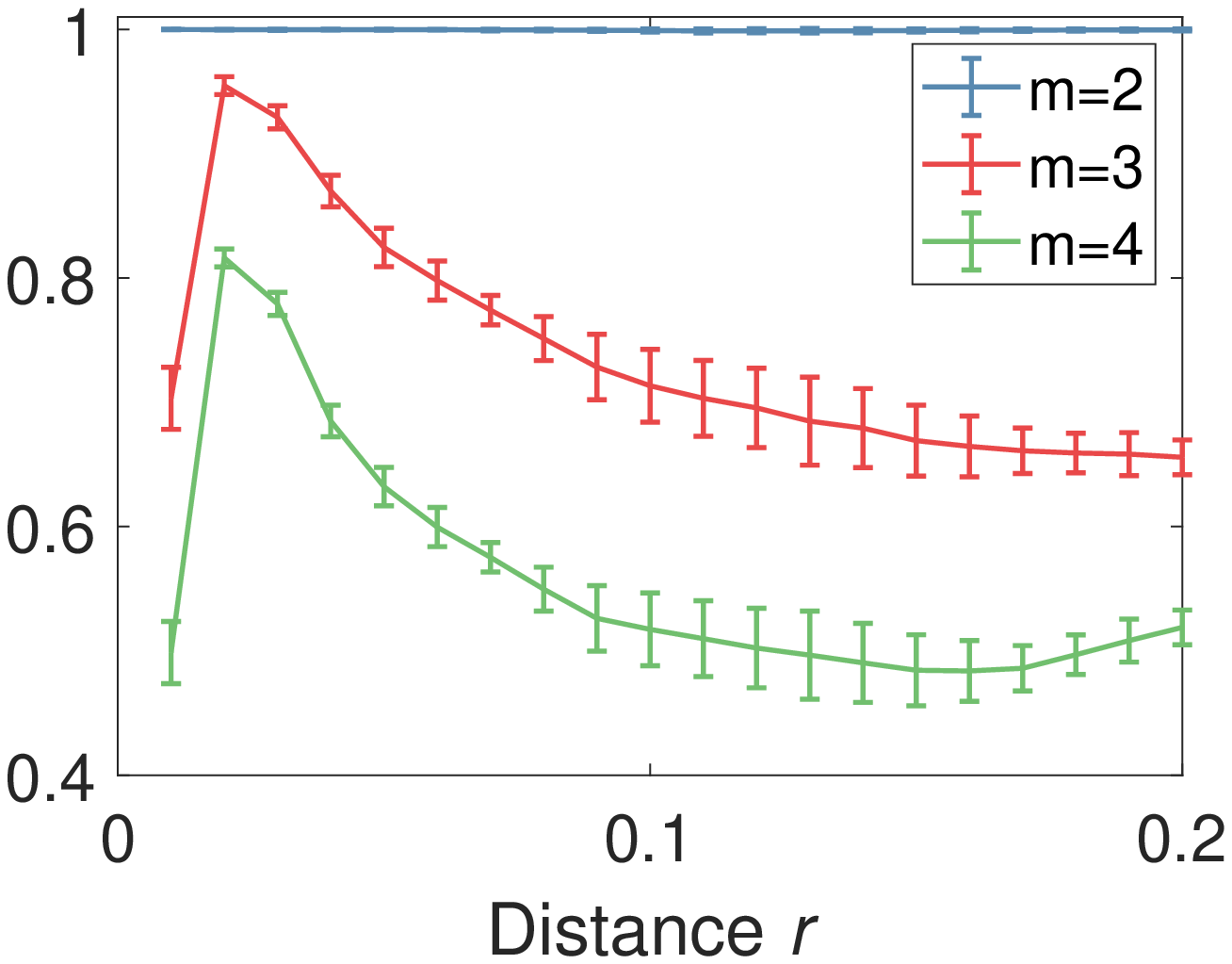}
		\caption{\label{fig:freq_p2PEG}} 
	\end{subfigure}
	\caption{\label{fig:freq_pPEG} Entropy values using $\PEG$ (a) for a fixed graph, increasing the noise and for several frequencies and (b) the underlying graph is more connected.}
\end{figure}

\textbf{The spectrum of the Laplacian and $\PEG$.} The entropy values of $\PEG$ exhibit a highly consistent and regular behaviour, with minimal variations (Fig.~\ref{fig:DEGeigenvaluesPEG}). Despite the varying degrees of irregularity in the eigenvalues (as shown in Fig.~4 of the manuscript), the $\PEG$ algorithm fails to detect these differences.
\begin{figure}[h]
	\centering
	\includegraphics[scale=.2]{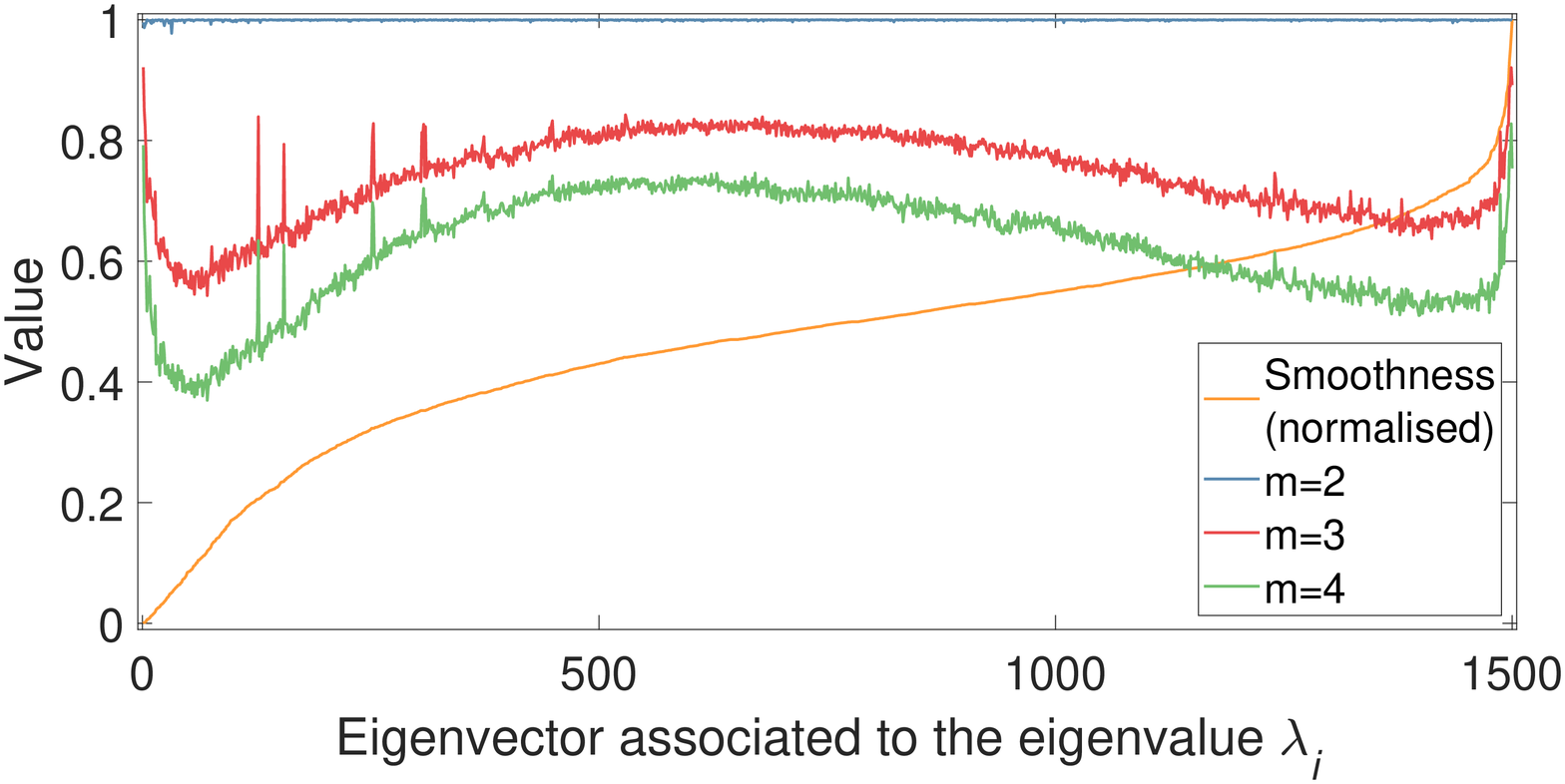} 
	\caption{Entropy values of $\PEG$ and smoothness based on the Laplacian $\Delta$ for the eigenvalues as graph signals.\label{fig:DEGeigenvaluesPEG}}
\end{figure}

\textbf{Small-world Networks and $\PEG$.} The stochastic dynamics of the Wiener process are not adequately characterized by $\PEG$ (Fig.~\ref{fig:timeseries_p1PEG}), as its entropy values are higher than those of random behaviour (random signal). Periodic dynamics are detected only with lower parameter values of $p$, and the chaotic and oscillation behaviour (Logistic map) are identified by $\PEG$, which is consistent with the results presented in~\cite{Fabila-Carrasco2022}. However, as the parameter $k$ is increased (Fig.~\ref{fig:timeseries_p2PEG}), the performance of $\PEG$ remains similar when the parameter $p$ is changed. This is due to $\PEG$ considering the order of the values but not their amplitude.
\begin{figure}
	\centering
	\begin{subfigure}[b]{0.49\linewidth}
		\includegraphics[scale=.3]{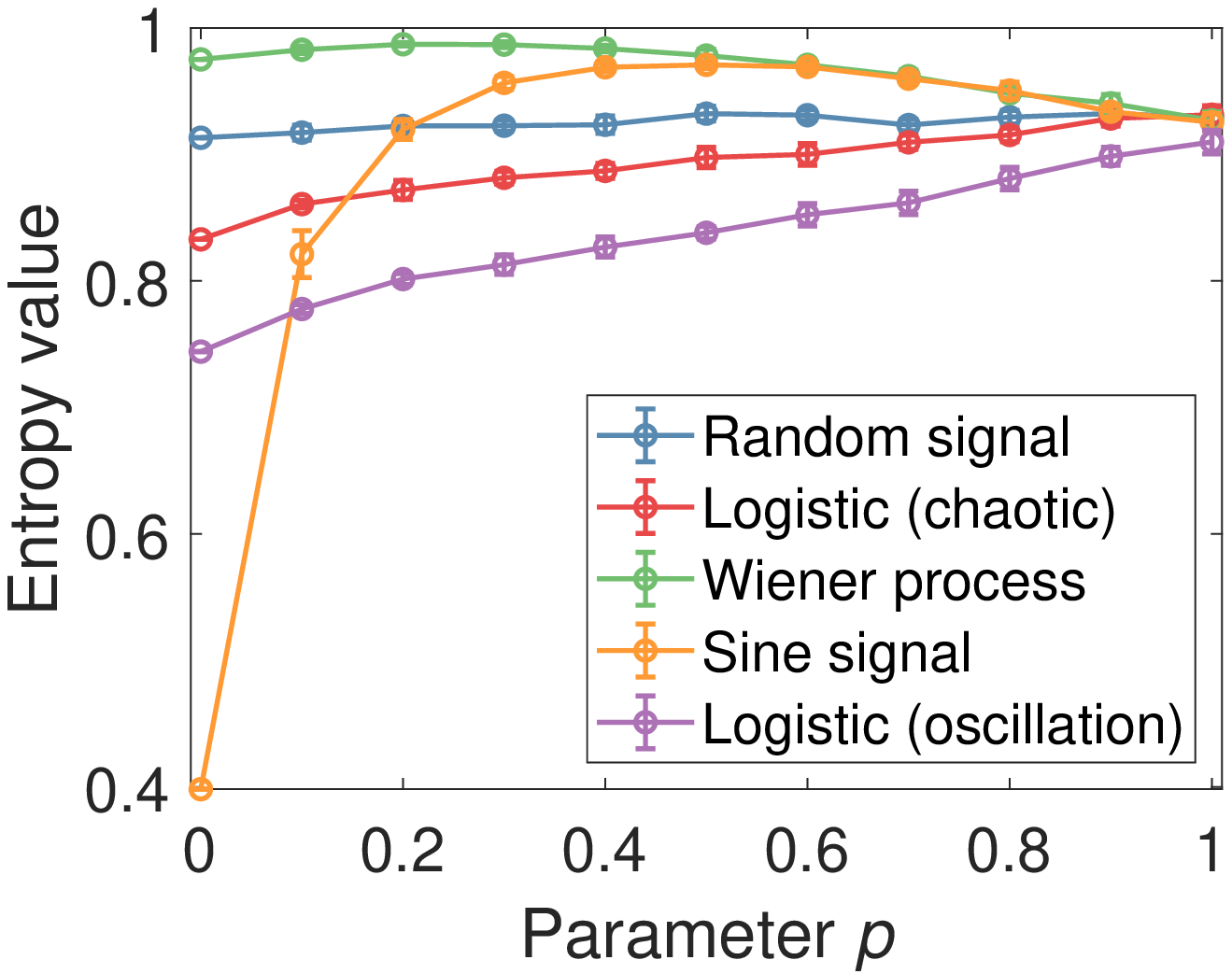} 
		\caption{\label{fig:timeseries_p1PEG} }
	\end{subfigure}
	\begin{subfigure}[b]{0.49\linewidth}
		\includegraphics[scale=.3]{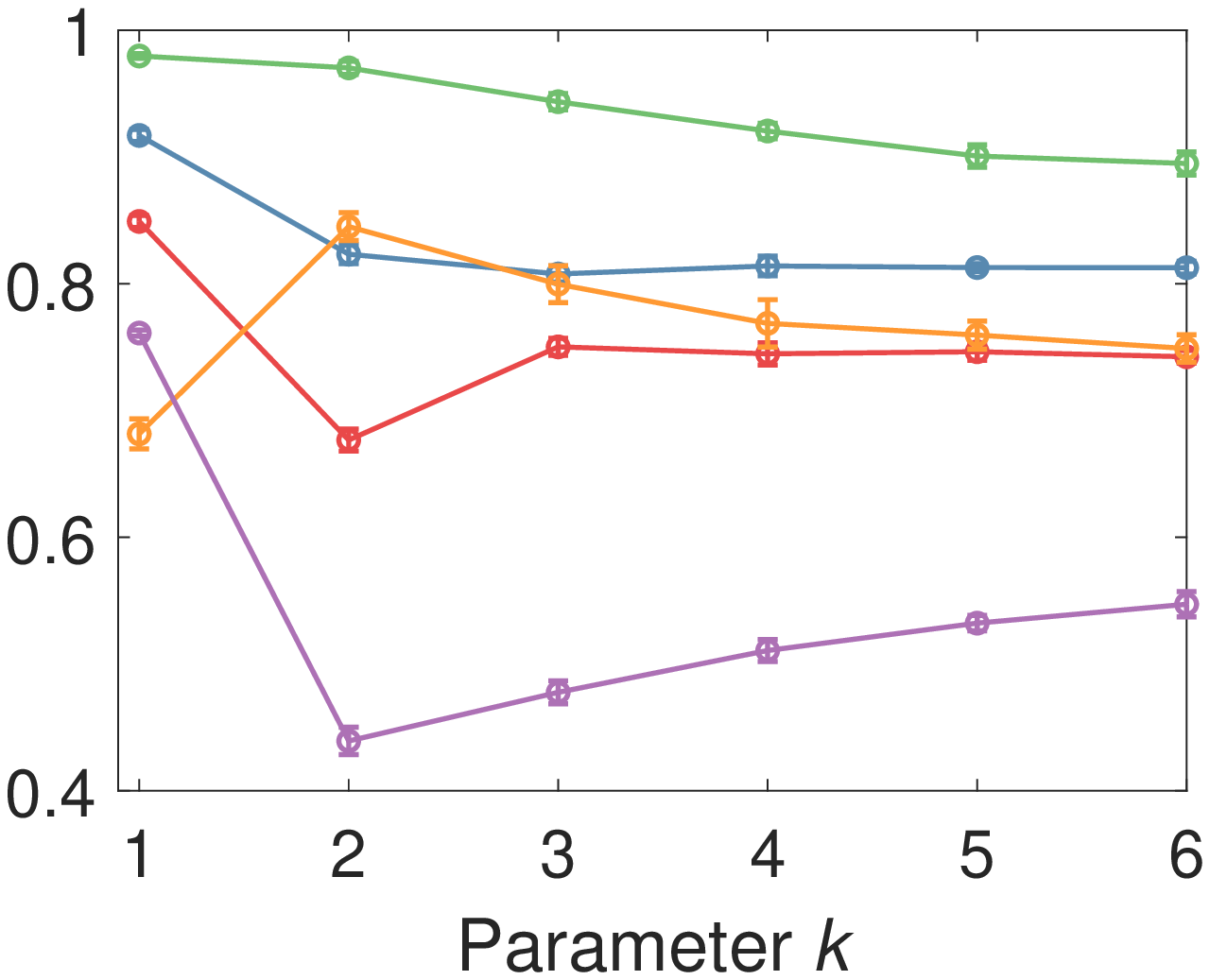}
		\caption{\label{fig:timeseries_p2PEG} } 
	\end{subfigure}
	\caption{\label{fig:timeseriesPEG} Entropy values of $\PEG$ for different signals defined on a small-world network generated by the Watts-Strogatz model.}
\end{figure}

\textbf{Graph Centrality Measures and $\PEG$.} Smooth signals produced by the \emph{Eigenvector Centrality} are not effectively detected by $\PEG$ (with the exception of the Arxiv and Facebook graphs). The remaining centrality measures yield similar entropy values, making it challenging to establish a relationship with $\PEG$ (Fig.~\ref{fig:central2PEG}). This limitation highlights the greater value of $\DEG$ for such analyses.
\begin{figure}
	\centering
	\includegraphics[scale=.2]{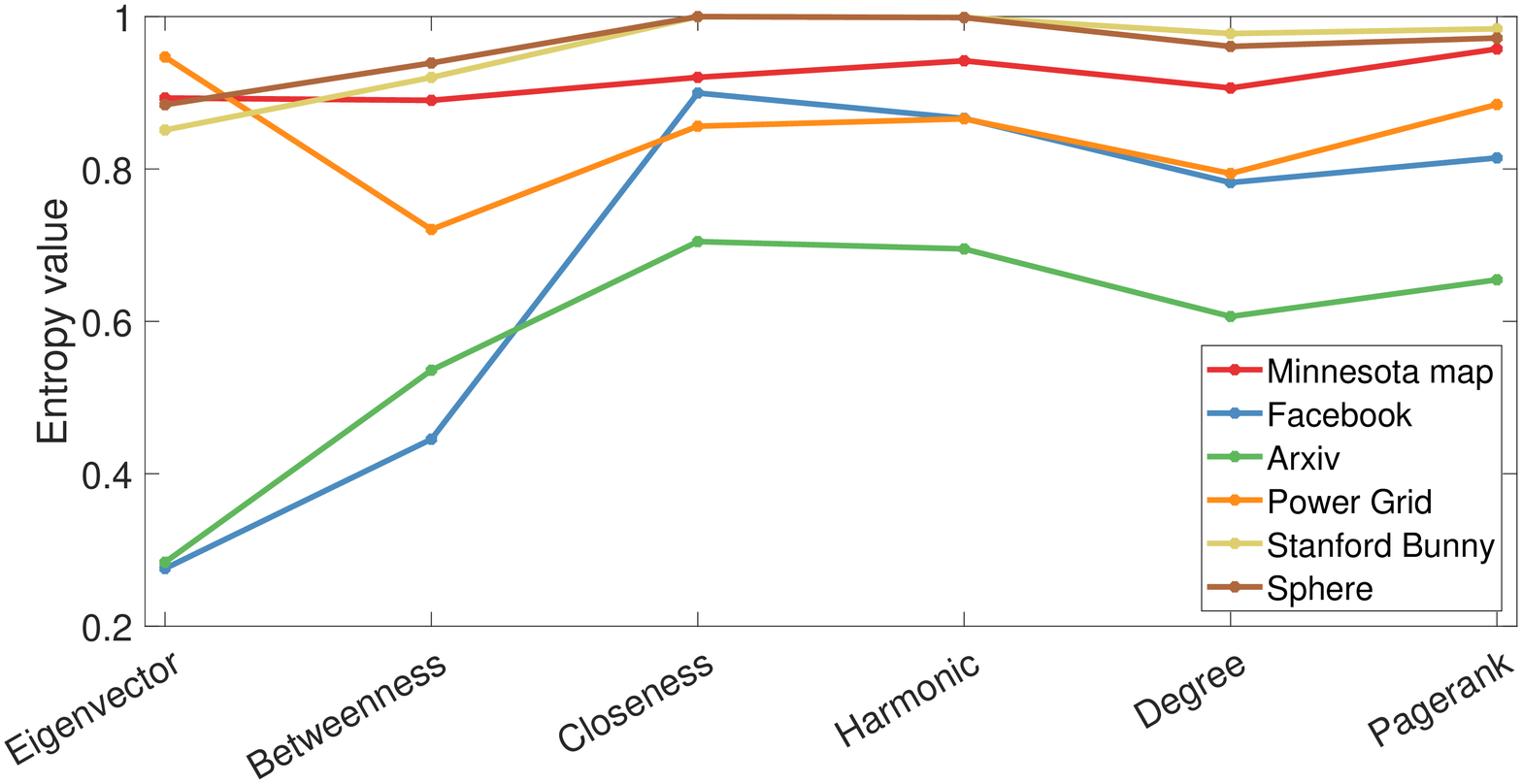}  
	\caption{\label{fig:central2PEG} The permutation entropy for various centrality measures.}
\end{figure}

\bibliography{Dispersion_Entropy_Arxiv}

\end{document}